\documentclass[twoside,12pt,a4paper]{amsart}
\usepackage{amscd,amsmath}
\usepackage{amssymb}
\usepackage{amsthm}
\usepackage{color}
\usepackage{hyperref}
\usepackage{enumerate}
\usepackage{geometry}
\theoremstyle{definition}
\newtheorem{theorem}{Theorem}[section]

\newtheorem{lemma}{Lemma}[section]

\numberwithin{equation}{section}

\def\<{\left < }
\def\>{\right >}
\def\({\left ( }
\def\){\right )}

\def\C2{${\bf C}^2$}

\begin{document}
 	\clearpage
 	\date{}

\title[Geom. ineq. on bi-warp. prod. subman. of locally conf. alm. cosympl.]{Geometric inequalities on bi-warped product submanifolds of locally conformal almost cosymplectic manifolds}

\author[R. Kaur]{Ramandeep KAUR}
\address{Ramandeep Kaur -  \textit{Department of Mathematics and Statistics, Central University of Punjab, Bathinda, Punjab-151 401, India}; \textit{E-mail Address: ramanaulakh1966@gmail.com}}
\author[G. Shanker]{Gauree SHANKER}
\address{Gauree Shanker - \textit{Department of Mathematics and Statistics, Central University of Punjab, Bathinda, Punjab-151 401, India}; \textit{E-mail Address: gauree.shanker@cup.edu.in}}
\author[A. Pigazzini]{Alexander PIGAZZINI$^*$}
\address{Alexander Pigazzini$^*$ - \textit{Mathematical and Physical Science Foundation, 4200 Slagelse, Denmark}; \textit{E-mail Address: pigazzini@topositus.com}}
\author[S. Jafari]{Saeid JAFARI}
\address{Saeid Jafari - \textit{Mathematical and Physical Science Foundation, 4200 Slagelse, Denmark and College of Vestsjaelland South, Herrestraede 11, 4200 Slagelse, Denmark}; \textit{E-mail Address: saeidjafari@topositus.com}}
\author[C. \"Ozel]{Cenap \"OZEL}
\address{Cenap \"Ozel - \textit{Department of Mathematics, Faculty of Science, King Abdulaziz University, 21589 Jeddah, Saudi Arabia}; \textit{E-mail Address: cozel@kau.edu.sa}}
\author[A. Mustafa]{Abdulqader MUSTAFA}
\address{Abdulqader Mustafa - \textit{Department of Mathematics, Faculty of Arts and Science, Palestine Technical University, Kadoorei, Tulkarm, Palestine}; \textit{E-mail Address: abdulqader.mustafa@ptuk.edu.ps}}

\maketitle
\begin{abstract}
In this paper we present not only some properties related to bi-warped product submanifolds of locally conformal almost cosymplectic manifolds, but also we show how the squared norm of the second fundamental form and the bi-warped product's warping functions are related  when the bi-warped product submanifold has a proper slant submanifold as a base or fiber.

\end{abstract}
 \section{Introduction}
The warped product is one of the most fruitful generalizations of the notion of the Cartesian product and it was introduced by Bishop and O'Neill in \cite{Bish}.
The interest in warped products has gradually grown more and more, as they allow to find many exact solutions to Einstein's field equation. Examples include the Robertson-Walker models and the Schwarzschild solution, the latter laying the foundation for describing the final stages of gravitational collapse and objects known today as black holes (see \cite{O'Neill}).
\\
Over the years many types of warped-products have been studied among others, for example the Einstein multiply warped-product manifolds, Einstein Warped-twisted product manifolds, Einstein sequential warped product manifolds which are those that cover the widest variety of exact solutions to the field equation of Einstein; Recently, in \cite{Pigaz}, Pigazzini et al. studied a special case of Einstein sequential warped-product manifolds, which cover a wider variety of exact solutions to Einstein's field equation, than Einstein warped-product manifolds with Ricci-flat fiber F, without complicating the calculations. In particular they studied a conformal semi-Riemannian Einstein metrics showing the existence of solutions with positive scalar curvature.
The extrinsic and intrinsic Riemannian invariants have broad applications in many different fields of science and differential geometry and they are also of considerable importance in general relativity \cite{Oss}. Among the extrinsic invariants, second fundamental form and the squared mean curvatures are the most important ones, while among the main intrinsic invariants, sectional, Ricci and scalar curvatures are the well-known ones, as well as $\delta$ invariant.
In \cite{BY-Chen}, Chen initiated a significant inequality in terms of the intrinsic invariant  ($\delta$-invariant) and more recently (in \cite{BYChen}), using Codazzi equation, he establishes an inequality for the second fundamental form in terms of warping function.
Similar inequality has been studied for warped product immersions in cosymplectic space forms, in \cite{Uddin}, by Uddin and Alqahtani. Over the years, the inequalities of submanifolds in warped-geometry have become the subject of study by many authors (among many, see for example \cite{AAA1}, \cite{AA3}, \cite{AA1}, \cite{C6}, \cite{MA2}, \cite{US1}, \cite{UU3}, \cite{UU2}).
\\
The purpose of the present paper is not only show some properties related to bi-warped product submanifolds of locally conformal almost cosymplectic manifolds, but also show how the squared norm of the second fundamental form and the bi-warped product's warping functions are related when the bi-warped product submanifold has a proper slant submanifold as a base or fiber.
\\
The paper is organized as follows: In the \textit{"Preliminaries"} section, we will address the definition of almost contact manifold and the relation with Slant manifolds. Subsequently, in the section \textit{"Bi-Warped Product Submanifolds of a locally conformal almost cosymplectic manifold"}, we will report some notions, theorems and lemmas that will be fundamental for the main result of the paper. The last section, namely \textit{"Chen type inequality for Bi-Warped product immersions"}, which contains the main result of this paper, i.e., we show how the square norm of the second fundamental form and the warping functions of the bi-warped product (of locally conformal almost cosymplectic manifolds), are correlated, carrying out a double analysis: when the  bi-warped product submanifold has a proper slant submanifold as a base or as a fiber.
 \section{Preliminaries}
 An almost contact manifold $\tilde{\mathbf{L}}$ of dimension $2m+1$ satisfies the following with contact structure $(\Phi, \xi, \eta,g)$ 
 \begin{align}\label{a1}
 \Phi^{2}&=-I+\eta \otimes \xi\notag\\
 \eta(\xi)&=1\notag\\
 \Phi\xi&=0\notag\\
 \eta\circ\Phi&=0\\
 \label{a2}
 g(\Phi {\mathbf{X}_4}, \Phi {\mathbf{Y}_4})&= g({\mathbf{X}_4},{\mathbf{Y}_4})- \eta ({\mathbf{X}_4}) \eta ({\mathbf{Y}_4}),
 \end{align}
 for any ${\mathbf{X}_4},{\mathbf{Y}_4} \in \Gamma (\mathbf{T}\tilde{\mathbf{L}}),$ where the $\Gamma(\mathbf{T}\tilde{\mathbf{L}})$ denotes the Lie algebra of vector fields on $\tilde{\mathbf{L}}$. 
 An almost contact manifold $\tilde{\mathbf{L}}$ reduce to be a locally conformal almost cosymplectic manifold if and only if \cite{R80}
 \begin{align}\label{a3}
 (\tilde{\nabla}_{{\mathbf{X}_4}}\Phi){\mathbf{Y}_4} = \beta \left(g(\Phi {\mathbf{X}_4},{\mathbf{Y}_4})\xi -\eta ({\mathbf{Y}_4})\Phi {\mathbf{X}_4}\right),~\tilde{\nabla}_{{\mathbf{X}_4}}\xi= \beta({\mathbf{X}_4}-\eta ({\mathbf{X}_4})\xi),
 \end{align}   
 for any ${\mathbf{X}_4},{\mathbf{Y}_4} \in \Gamma(\mathbf{T}\tilde{\mathbf{L}}),$ where $\tilde{\nabla}$ denotes the Levi-Civita connection of $g.$\\
 
 The Gauss and Weingarten formulas for a submanifold $M$ in a Riemannian manifold $\tilde{\mathbf{L}}.$ are given by
 \begin{align}\label{a4}
 \tilde{\nabla}_{{\mathbf{X}_4}}{\mathbf{Y}_4}= \nabla_{{\mathbf{X}_4}}{\mathbf{Y}_4} +\Sigma ({\mathbf{X}_4},{\mathbf{Y}_4}),
 \end{align}
 \begin{align}\label{a5}
 \tilde{\nabla}_{{\mathbf{X}_4}}N= -A_{N}{\mathbf{X}_4} + \nabla_{{\mathbf{X}_4}}^{\perp}N,
 \end{align} 
 for any vector field ${\mathbf{X}_4},{\mathbf{Y}_4} \in \Gamma(\mathbf{TM})$ and $N\in \Gamma (\mathbf{T}^{\perp}\mathbf{L})$. The shape operator and second funadmental form are denoted by $A$ and $\Sigma$, respectively, and have following  relation 
 \begin{align}\label{a6}
 g(\Sigma({\mathbf{X}_4},{\mathbf{Y}_4}),N)=g(A_{N}{\mathbf{X}_4},{\mathbf{Y}_4}).
 \end{align}
 A submanifold $\mathbf{L}$ is said to be totally geodesic if $\Sigma=0$ and totally umbilical if $\Sigma({\mathbf{X}_4},{\mathbf{Y}_4})=g({\mathbf{X}_4},{\mathbf{Y}_4})H,~\forall {\mathbf{X}_4},{\mathbf{Y}_4} \in \Gamma(\mathbf{TM}),$ where $H=\frac{1}{n}\Sigma_{i=1}^{n}\Sigma(e_{i},e_{i})$ is the mean curvature vector of $\mathbf{L}.$ For any $x\in \mathbf{L}, \{e_{1},\dots,e_{n},\dots,e_{2m+1}\}$ is an orthonormal frame of the tangent space $\mathbf{T}_{x}\tilde{\mathbf{L}}$ such that $\{e_{1},\dots,e_{n}\}$ are tangent to $\mathbf{L}$ at $x.$ Then, we have 
 \begin{align}\label{a7}
 \Sigma_{ij}^{r}=g(\Sigma(e_{i},e_{j}),e_{r})
 \end{align}
 \begin{align}\label{a8}
 ||\Sigma||^{2}= \sum_{i,j=1}^{n} g(\Sigma(e_{i},e_{j}),\Sigma(e_{i},e_{j})).
 \end{align}
 for any $1\leq i, j\leq n$ and $1\leq r\leq 2m+1$. Then length of the gradient $\vec{\nabla}(\mathbf{\phi})$ for a differentiable function $\mathbf{\phi}$ manifold ${\mathbf{L}}$ is defined as
 \begin{align}\label{a9}
 ||\vec{\nabla}(\mathbf{\phi})||^{2}=\sum_{1}^{n}(e_{i}(\mathbf{\phi}))^{2},
 \end{align}
 for the orthonormal frame $\{e_{i},\dots,e_{n}\}$ on $\tilde{\mathbf{L}}.$  The tangential $\mathbf{TX}$ and normal components $\mathbf{FX}$ of $\Phi {\mathbf{X}_4}$ for any vector field ${\mathbf{X}_4}\in \Gamma(\mathbf{TM}),$ are decomposed as
 \begin{align}\label{a10}
 \Phi {\mathbf{X}_4}=\mathbb{\mathbf{T}}{\mathbf{X}_4}+\mathbb{\mathbf{\phi}}{\mathbf{X}_4}.
 \end{align}
 \begin{enumerate}
 	\item [(1)] A submanifold $\mathbf{L}$ tangent to $\xi$ is said to be slant if for any $x \in \mathbf{L}$ and any $X\in \mathbf{T}_{x}\mathbf{L},$ linearly independent to $\xi,$ the angle between $\Phi X$ and $\mathbf{T}_{x}\mathbf{L}$ is a constant $\nu \in [0,\pi/2],$ called the slant angle of $\mathbf{L}$ in $\tilde{\mathbf{L}}.$
 	\item[(2)] Invariant and anti-invariant submanifolds are $\nu$-slant submanifolds with slant angle $\nu=0$ and $\nu=\pi/2,$ respectively. Proper slant is a slant submanifold that is neither invariant nor anti-invariant..
 \end{enumerate}
 For slant submanifolds, we have the following characterisation theorem.

 \begin{theorem}\label{t21}
 	Let $\textbf{M}$ be a submanifold of an almost contact metric manifold $\tilde{M},$ such that $\xi \in \Gamma(\mathbf{TM}).$ Then $\mathbf{L}$ is slant if and only if there exists a constant $\lambda \in [0,1]$ such that
 	\begin{align}\label{a11}
 	\mathbf{T}^{2}= \lambda(-I+ \eta \otimes \xi).
 	\end{align}
 	Moreover, if $\nu$ is slant angle, then $\lambda=\cos^{2}\nu.$ \\
 \end{theorem}
 The following relationships are a straightforward result of \eqref{a11}.
 \begin{align}\label{a12}
 g(\mathbf{TX},\mathbf{TY})=\cos^{2}\nu\bigg\{g({\mathbf{X}_4},{\mathbf{Y}_4})- \eta ({\mathbf{X}_4})\eta ({\mathbf{Y}_4})\bigg\},
 \end{align}
 \begin{align}\label{a13}
 g(\mathbf{FX}, \mathbf{FY})=\sin^{2}\nu\bigg\{g({\mathbf{X}_4},{\mathbf{Y}_4})- \eta ({\mathbf{X}_4})\eta ({\mathbf{Y}_4})\bigg\},
 \end{align}
 for any ${\mathbf{X}_4},{\mathbf{Y}_4} \in \Gamma(\mathbf{TM}).$
 
 \section{Bi-Warped Product Submanifolds of a locally conformal almost cosymplectic manifold}
 Let $\mathbf{L}=\mathbf{L}_{1}\times \mathbf{L}_{2}\times \mathbf{L}_{3}$ be the Cartesian product of Riemannian manifolds $\mathbf{L}_{1},\,\,\mathbf{L}_{2}$ and $\mathbf{L}_{3}$. The canonical projections of $\mathbf{L}$ onto $\mathbf{L}_{i}$ are denoted and defined by $\pi_{i}:\mathbf{L}\rightarrow \mathbf{L}_{i}$ for each $i=1,2,3.$ If $\mathbf{\phi}_{2},\mathbf{\phi}_{3}:\mathbf{L}_{1}\rightarrow \mathbb{R}^{+}$ are positive real valued functions, then
 $$g({\mathbf{X}_4},{\mathbf{Y}_4})=g(\pi_{1*}{\mathbf{X}_4},\pi_{1*}{\mathbf{Y}_4})+(\mathbf{\phi}_{2}\circ \pi_{1})^{2}g(\pi_{2*}\mathbf{X_4},\pi_{2*}\mathbf{Y_4})+ (\mathbf{\phi}_{3}\circ \pi_{1})^{2}g(\pi_{3*}{\mathbf{X}_4},\pi_{3*}{\mathbf{Y}_4}),$$
 defines a Riemannian metric $g$ on $\mathbf{L}$. A product manifold $\mathbf{L}$ endowed with metric $g$ is called a bi-warped product manifold. In this case, $\mathbf{\phi}_{2}, \mathbf{\phi}_{3}$ are non-constant functions, called warping functions on $\mathbf{L}$. $\mathbf{L}$ is a simply Riemannian product manifold $\mathbf{L}=\mathbf{L}_{1}\times \mathbf{L}_{2}\times \mathbf{L}_{3}$ if both $\mathbf{\phi}_{2},\mathbf{\phi}_{3}$ are constant on $\mathbf{L}.$ 
 Let $\mathbf{L}=\mathbf{L}_{1}\times _{\mathbf{\phi}_{2}} \mathbf{L}_{2}\times _{\mathbf{\phi}_{3}} \mathbf{L}_{3}$ be a bi-warped product submanifold of a Riemannian manifold $\tilde{\mathbf{L}}.$ Then, we have
 \begin{align}\label{b1}
 \nabla_{{\mathbf{X}_4}}{\mathbf{Z}_4}=\sum_{i=2}^{3}({\mathbf{X}_4}(\ln \mathbf{\phi}_{i})){\mathbf{Z}_4}^{i},
 \end{align}
 for any ${\mathbf{X}_4}\in \mathfrak{D}_{1},$ the tangent space of $\mathbf{L}_{1}$ and ${\mathbf{Z}_4}\in \mathbf{TN},$ where $\mathbf{N}= _{\mathbf{\phi}_{2}} \mathbf{L}_{2}\times _{\mathbf{\phi}_{3}} \mathbf{L}_{3}$ and ${\mathbf{Z}_4}^{i}$ is $\mathbf{L}_{i}$-component of ${\mathbf{Z}_4}$ and $\nabla$ is Levi-Civita connection on $\mathbf{L}.$\\
 
 Let $\mathbf{\phi}_{1},\mathbf{\phi}_{2}:\mathbf{L}_{\nu}\rightarrow \mathbb{R}^{+}$ be non-constant functions. Then we consider the bi-warped product submanifolds of the form $\mathbf{L}= \mathbf{L}_{\nu}\times_{\mathbf{\phi}_{1}} \mathbf{L}_{\mathbf{L}T} \times _{\mathbf{\phi}_{2}} \mathbf{L}_{\perp}$ in a locally conformal almost cosymplectic manifold $\tilde{\mathbf{L}},$ where $\mathbf{L}_{\mathbf{T}},\mathbf{L}_{\perp}$ and $\mathbf{L}_{\nu}$ are invariant, anti-invariant and proper slant submanifolds of $\tilde{\mathbf{L}},$ respectively. In this case, the tangent and normal space of $\mathbf{L}$ are decomposed as according to the integrals manifolds $\mathbf{L}_{\mathbf{T}},\,\,\mathbf{L}_{\perp}$ and $\mathbf{L}_{\nu}$ of $\mathfrak{D},\,\,\mathfrak{D}^{\perp}$ and $\mathfrak{D}^{\nu},$ respectively.
 
 \begin{align}\label{b2}
 \mathbf{TM}= \mathfrak{D}\oplus \mathfrak{D}^{\perp}\oplus\mathfrak{D}^{\nu}\oplus \langle \xi \rangle
 \end{align}
 and 
 \begin{align}\label{b3}
 \mathbf{T}^{\perp}\mathbf{L}=\Phi \mathfrak{D}^{\perp}\oplus \mathbf{\phi}\mathfrak{D}^{\nu}\oplus \mu, ~ \Phi \mathfrak{D}^{\perp}\oplus \mathbf{\phi}\mathfrak{D}^{\nu}\perp \mu,
 \end{align}
 where $\mu$ is an $\Phi$-invariant normal subbundle of $\mathbf{T}^{\perp}\mathbf{L}.$ Now we obtain some classification theorems as follows
 \begin{theorem}
 	If $\xi$ is tangent to either $\xi \in \Gamma(\mathfrak{D})$ or $\xi \in \Gamma(\mathfrak{D}^{\perp}),$ then  a bi-warped product submanifold of the type $\mathbf{L}=\mathbf{L}_{\nu}\times_{\mathbf{\phi}_{1}} \mathbf{L}_{\mathbf{T}} \times _{\mathbf{\phi}_{2}} \mathbf{L}_{\perp}$ in a locally conformal almost cosymplectic manifold $\tilde{\mathbf{L}}$ is a single warped product.	
 \end{theorem}
 \begin{proof}
 	If $\xi \in \Gamma(\mathfrak{D}),$ then for any $\mathbf{U_4} \in \Gamma(\mathfrak{D}),$ we get
 	\begin{align*}
 	\tilde{\nabla}_{\mathbf{U_4}} \xi =\beta \mathbf{U_4}.
 	\end{align*}
 	Using \eqref{a4} and \eqref{b1}, we obtain	
 	\begin{align*}
 	\mathbf{U_4}(\ln \mathbf{\phi}_{1})\xi =\beta \mathbf{U_4}.
 	\end{align*}
 	We get the following by multiplying $\xi$ by the inner product and utilizing the fact that $\xi \in \Gamma(\mathfrak{D}),$ 
 	\begin{align*}
 	\mathbf{U_4}(\ln \mathbf{\phi}_{1})=0.
 	\end{align*}
 	As a result, $\mathbf{\phi}_{1}$ is constant, and so $\mathbf{L}$ is a warped product manifold. In the same way, we may simply achieve that.
 	\begin{align*}
 	\mathbf{U_4}(\ln \mathbf{\phi}_{2})=0,
 	\end{align*}
 	which implies $\mathbf{\phi}_{2}$ is constant and thus the proof.
 \end{proof}
 \begin{theorem}
 	Let $\xi$ is tangent to $\mathbf{L}_{\nu}$ on a bi-warped product submanifold of the type $\mathbf{L}=\mathbf{L}_{\nu}\times_{\mathbf{\phi}_{1}} \mathbf{L}_{T} \times _{\mathbf{\phi}_{2}} \mathbf{L}_{\perp}$ in a  locally conformal almost cosymplectic manifold $\tilde{\mathbf{L}},$. Then we have
 	\begin{align}\label{b4}
 	\xi (\ln \mathbf{\phi}_{i})=\beta,\,\,\,\,\forall~ i=1,2.
 	\end{align}
 \end{theorem}
 \begin{proof}
 	For any ${\mathbf{X}_4} \in \Gamma(\mathfrak{D}),$ we have 
 	$$\tilde{\nabla}_{{\mathbf{X}_4}} \xi =\beta {\mathbf{X}_4}.$$ 
 	Utilizing \eqref{a4} and \eqref{b1}, we obtain	
 	\begin{align*}
 	\xi (\ln \mathbf{\phi}_{1}){\mathbf{X}_4}=\beta {\mathbf{X}_4}.
 	\end{align*}
 	Taking the inner product with ${\mathbf{X}_4}$, we derive
 	\begin{align*}
 	\xi (\ln \mathbf{\phi}_{1})=\beta.
 	\end{align*}
 	Similarily, it can be obtain
 	\begin{align*}
 	\xi (\ln \mathbf{\phi}_{2})=\beta.
 	\end{align*}
 \end{proof}
 \begin{lemma}\label{l33}
 	Let $\mathbf{L}=\mathbf{L}_{\nu}\times_{\mathbf{\phi}_{1}} \mathbf{L}_{\mathbf{T}} \times _{\mathbf{\phi}_{2}} \mathbf{L}_{\perp}$ be a bi-warped product submanifold of locally conformal almost cosymplectic manifold $\tilde{\mathbf{L}}.$ Then, we have 
 	\begin{itemize}
 		\item [(i)] $g(\Sigma({\mathbf{X}_4},{\mathbf{Y}_4}), \mathbf{FU_4})=\Big((\mathbf{U_4}\ln \mathbf{\phi}_{1})-\beta\eta(\mathbf{U_4})\Big)g({\mathbf{X}_4},\Phi {\mathbf{Y}_4})+ \mathbf{TU_4}(\ln \mathbf{\phi}_{1})g({\mathbf{X}_4},{\mathbf{Y}_4}),$
 		\item [(ii)] $g(\Sigma({\mathbf{Z}_4},\mathbf{W_4}),\mathbf{FU_4})= \beta g(\Sigma(\mathbf{U_4},\mathbf{W_4}),\Phi {\mathbf{Z}_4})+\mathbf{TU_4}(\ln \mathbf{\phi}_{2})g({\mathbf{Z}_4},\mathbf{W_4}),$
 		\item [(iii)] $g(\Sigma(\mathbf{U_4},\mathbf{V_4}),\Phi {\mathbf{Z}_4})=g(\Sigma(\mathbf{U_4},{\mathbf{Z}_4}),\mathbf{FV}),$
 		\item [(iv)] $g(\Sigma({\mathbf{X}_4},\mathbf{V_4}),\mathbf{FU_4})=0,$\\
 		for any ${\mathbf{X}_4},{\mathbf{Y}_4}\in \Gamma(\mathfrak{D}),~ \mathbf{U_4},\mathbf{V_4}\in \Gamma(\mathfrak{D}^{\nu}\oplus\xi)$ and ${\mathbf{Z}_4},\mathbf{W_4}\in \Gamma(\mathfrak{D}^{\perp}).$
 	\end{itemize}
 \end{lemma}
 \begin{proof}
 	For any ${\mathbf{X}_4},{\mathbf{Y}_4}\in \Gamma(\mathfrak{D})$ and  $\mathbf{U_4}\in \Gamma(\mathfrak{D}^{\nu}\oplus\xi)$ we have
 	\begin{equation*}
 	\begin{split}
 	g(\Sigma({\mathbf{X}_4},{\mathbf{Y}_4}),\mathbf{FU_4})&=g(\tilde{\nabla}_{{\mathbf{X}_4}}{\mathbf{Y}_4},\mathbf{FU_4})\\
 	&=g(\tilde{\nabla}_{{\mathbf{X}_4}}{\mathbf{Y}_4},\Phi \mathbf{U_4})-(\tilde{\nabla}_{{\mathbf{X}_4}}{\mathbf{Y}_4},\mathbf{TU_4})\\
 	&=g((\tilde{\nabla}_{{\mathbf{X}_4}}\Phi){\mathbf{Y}_4},\mathbf{U_4})-g(\tilde{\nabla}_{{\mathbf{X}_4}}\Phi {\mathbf{Y}_4},\mathbf{U_4})+g({\mathbf{Y}_4},\tilde{\nabla}_{{\mathbf{X}_4}}\mathbf{TU_4}).
 	\end{split}
 	\end{equation*} 
 	Using \eqref{a3} and \eqref{b1}, we get	
 	\begin{align*}
 	g(\Sigma({\mathbf{X}_4},{\mathbf{Y}_4}), \mathbf{FU_4})=\Big((\mathbf{U_4}\ln \mathbf{\phi}_{1})-\beta\eta(\mathbf{U_4})\Big)g({\mathbf{X}_4}, \Phi {\mathbf{Y}_4}) +\mathbf{TU_4}(\ln \mathbf{\phi}_{1})g({\mathbf{X}_4}, {\mathbf{Y}_4}).
 	\end{align*}
 	This is a first part. For the second, we have
 	\begin{align*}
 	g(\Sigma({\mathbf{Z}_4}, \mathbf{W_4}), \mathbf{FU_4})=g(\tilde{\nabla}_{{\mathbf{Z}_4}}\mathbf{W_4}, \Phi \mathbf{U_4})-g(\tilde{\nabla}_{{\mathbf{Z}_4}}\mathbf{W_4}, \mathbf{FU_4}),
 	\end{align*}
 	for any ${\mathbf{Z}_4},\mathbf{W_4}\in \Gamma(\mathfrak{D}^{\perp}).$  From the virtue \eqref{a2}, we obtain
 	\begin{equation*}
 	\begin{split}
 	g(\Sigma({\mathbf{Z}_4},\mathbf{W_4}),\mathbf{FU_4})&=g((\tilde{\nabla}_{{\mathbf{Z}_4}}\Phi)\mathbf{W_4},\mathbf{U_4})-g(\tilde{\nabla}_\textbf{Z}{\Phi}\mathbf{W_4}, \mathbf{U_4})-g(\mathbf{W_4},\tilde{\nabla}_{{\mathbf{Z}_4}}\mathbf{FU_4}).
 	\end{split}
 	\end{equation*}
 	Making use of \eqref{a3}, \eqref{a4} and \eqref{b1}, we get
 	\begin{align*}
 	g(\Sigma({\mathbf{Z}_4},\mathbf{W_4}), \mathbf{FU_4})=(g(\Sigma(\mathbf{U_4}, \mathbf{W_4})), \Phi {\mathbf{Z}_4})+\mathbf{TU_4}(\ln \mathbf{\phi}_{2})g({\mathbf{Z}_4}, \mathbf{W_4}).
 	\end{align*}
 	Again $\mathbf{U_4},\mathbf{V_4}\in \Gamma(\mathfrak{D}^{\nu}\oplus\xi)\,\,\,\&\,\,\,\,{\mathbf{Z}_4}\in \Gamma(\mathfrak{D}^{\perp}).$ We derive
 	\begin{equation*}
 	g(\Sigma(\mathbf{U_4},\mathbf{V_4}),\Phi {\mathbf{Z}_4})=g(\tilde{\nabla}_{\mathbf{U_4}}\mathbf{V_4},\Phi {\mathbf{Z}_4})=g((\tilde{\nabla}_\mathbf{U_4}\Phi)\mathbf{V_4}, {\mathbf{Z}_4})-g(\tilde{\nabla}_\mathbf{U_4}\Phi \mathbf{V_4}, {\mathbf{Z}_4}).
 	\end{equation*}
 	The term $ g((\tilde{\nabla}_\mathbf{U_4}\Phi)\mathbf{V_4}, {\mathbf{Z}_4})=0$ using \eqref{a3}. Then by the orthogonality of vector fields, we obtain
 	\begin{equation*}
 	g(\Sigma(\mathbf{U_4},\mathbf{V_4}), \Phi {\mathbf{Z}_4})=g(\tilde{\nabla}_{\mathbf{U_4}}{\mathbf{Z}_4},\Phi\mathbf{V_4})=g(\nabla_{\mathbf{U_4}}{\mathbf{Z}_4},\mathbf{FV})+g(\tilde \nabla_{\mathbf{U_4}}{\mathbf{Z}_4},\mathbf{FV}),
 	\end{equation*}
 	Eqs \eqref{a5} and \eqref{b1} imply the following 
 	\begin{align*}
 	g(\Sigma(\mathbf{U_4}, \mathbf{V_4}),\Phi {\mathbf{Z}_4})=g(\Sigma(\mathbf{U_4}, {\mathbf{Z}_4}), \mathbf{FV}).
 	\end{align*}
 	Third is completed. For the last part, we have
 	\begin{equation*}
 	\begin{split}
 	g(\Sigma({\mathbf{X}_4}, \mathbf{V_4}), \mathbf{FU_4})&=g(\tilde{\nabla}_{\mathbf{V_4}}{\mathbf{X}_4},\Phi \mathbf{U_4})-g(\tilde{\nabla}_{\mathbf{V_4}}{\mathbf{X}_4}, \mathbf{TU_4}),\\
 	&=-g(\tilde{\nabla}_\mathbf{V_4}\Phi {\mathbf{X}_4}, \mathbf{U_4})+g((\tilde{\nabla}_{\mathbf{V_4}}\Phi){\mathbf{X}_4}, \mathbf{U_4})-\mathbf{V_4}(\ln \mathbf{\phi}_{1})g({\mathbf{X}_4}, \mathbf{TU_4}).
 	\end{split}
 	\end{equation*}
 	for any ${\mathbf{X}_4}\in \Gamma(\mathfrak{D})$ and $\mathbf{U_4},\mathbf{V_4}\in \Gamma(\mathfrak{D}^{\nu}\oplus\xi).$
 	Using \eqref{a3}, \eqref{b1}, we get \\
 	\begin{align*}
 	g(\Sigma({\mathbf{X}_4}, \mathbf{V_4}), \mathbf{FU_4})=0.
 	\end{align*}
 	So, the required results are obtained.
 \end{proof}
 If we interchange ${\mathbf{X}_4}$ by $\Phi {\mathbf{X}_4}$ and ${\mathbf{Y}_4}$ by $\Phi {\mathbf{Y}_4}$ in  (i) of Lemma $\ref{l33}$, then we find the following relations.
 \begin{align}\label{b5}
 g(\Sigma(\Phi {\mathbf{X}_4},{\mathbf{Y}_4}),\mathbf{FU_4})=\Big\{(\mathbf{U_4}\ln \mathbf{\phi}_{1})-\beta\eta(\mathbf{U_4})\Big\})g({\mathbf{X}_4}, {\mathbf{Y}_4})-\mathbf{TU_4}(\ln \mathbf{\phi}_{1})g({\mathbf{X}_4},\Phi {\mathbf{Y}_4}).
 \end{align}
 \begin{align}\label{b6}
 g(\Sigma({\mathbf{X}_4},\Phi {\mathbf{Y}_4}),\mathbf{FU_4})=-\Big\{(\mathbf{U_4}\ln \mathbf{\phi}_{1})-\beta\eta(\mathbf{U_4})\Big\}g({\mathbf{X}_4}, {\mathbf{Y}_4})+\mathbf{TU_4}(\ln \mathbf{\phi}_{1})g({\mathbf{X}_4},\Phi {\mathbf{Y}_4}).
 \end{align}
 and \\
 \begin{align}\label{b7}
 g(\Sigma(\Phi {\mathbf{X}_4},\Phi {\mathbf{Y}_4}), \mathbf{FU_4})=\Big\{(\mathbf{U_4}\ln \mathbf{\phi}_{1})-\beta\eta(\mathbf{U_4})\Big\}g({\mathbf{X}_4}, \Phi {\mathbf{Y}_4})+\mathbf{TU_4}(\ln \mathbf{\phi}_{1})g({\mathbf{X}_4}, {\mathbf{Y}_4}).
 \end{align}
 Similarly, we interchange $\mathbf{U_4}$ by $\mathbf{TU_4}$ in Lemma $\ref{l33}$ and in \eqref{b5}-\eqref{b7}, and using Theorem \ref{t21} and \eqref{b4}, we derive
 \begin{align}\label{b8}
 g(\Sigma({\mathbf{X}_4},{\mathbf{Y}_4}), \mathbf{FTU})=\mathbf{TU_4}(\ln \mathbf{\phi}_{1})g({\mathbf{X}_4},\Phi {\mathbf{Y}_4})-\cos^{2}\nu\Big(\mathbf{U_4}(\ln \mathbf{\phi}_{1})-\beta\eta(\mathbf{U_4})\Big)g({\mathbf{X}_4},{\mathbf{Y}_4})
 \end{align}
 \begin{align}\label{b9}
 g(\Sigma(\Phi {\mathbf{X}_4},{\mathbf{Y}_4}),\mathbf{FTU})=\mathbf{TU_4}(\ln \mathbf{\phi}_{1})g({\mathbf{X}_4},{\mathbf{Y}_4})+\cos^{2}\nu \Big(\mathbf{U_4}(\ln \mathbf{\phi}_{1})-\beta\eta(\mathbf{U_4})\Big)g({\mathbf{X}_4},\Phi {\mathbf{Y}_4})
 \end{align}
 \begin{align}\label{b10}
 g(\Sigma({\mathbf{X}_4},\Phi {\mathbf{Y}_4}), \mathbf{FTU})=-\mathbf{TU_4}(\ln \mathbf{\phi}_{1})g({\mathbf{X}_4},{\mathbf{Y}_4})-\cos^{2}\nu \Big(\mathbf{U_4}(\ln \mathbf{\phi}_{1})-\beta\eta(\mathbf{U_4})\Big)g({\mathbf{X}_4},\Phi {\mathbf{Y}_4}),
 \end{align}
 and 
 \begin{align}\label{b11}
 g(\Sigma(\Phi {\mathbf{X}_4},\Phi {\mathbf{Y}_4}), \mathbf{FTU})=\mathbf{TU_4}(\ln \mathbf{\phi}_{1})g({\mathbf{X}_4},\Phi {\mathbf{Y}_4})-\cos^{2}\nu\Big(\mathbf{U_4}(\ln \mathbf{\phi}_{1})-\beta\eta(\mathbf{U_4})\Big)g({\mathbf{X}_4}, {\mathbf{Y}_4}).
 \end{align}
 On the other hand, we interchange $\mathbf{U_4}$ by $\mathbf{TU_4}$ in Lemma $3.1(ii)$ and using \textit{Theorem 2.1} and \eqref{b4}, we obtain
 \begin{align}\label{b12}
 g(\Sigma({\mathbf{Z}_4},\mathbf{W_4}), \mathbf{FTU})=g(\Sigma(\mathbf{TU_4},\mathbf{W_4}),\Phi {\mathbf{Z}_4})-\cos^{2}\nu (\mathbf{V_4}(\ln \mathbf{\phi}_{2})-\eta (\mathbf{U_4}))g({\mathbf{Z}_4},\mathbf{W_4}).
 \end{align}
 \begin{lemma}
 	Let $\mathbf{L}=\mathbf{L}_{\nu}\times_{\mathbf{\phi}_{1}} \mathbf{L}_{\mathbf{UT}}\times_{\mathbf{\phi}_{2}}\mathbf{L}_{\perp}$ be a bi-warped product submanifold of a locally conformal almost cosymplectic manifold $\tilde{\mathbf{L}}$. Then, we have
 	\begin{align}\label{b13}
 	g(\Sigma({\mathbf{X}_4},\mathbf{V_4}),\Phi {\mathbf{Z}_4})=g(\Sigma({\mathbf{X}_4},{\mathbf{Z}_4}),\mathbf{FV})=0,
 	\end{align}
 	for any ${\mathbf{X}_4}\in \Gamma(\mathfrak{D}), ~\mathbf{V_4}\in \Gamma(\mathfrak{D}^{\nu})$ and ${\mathbf{Z}_4}\in \Gamma(\mathfrak{D}^{\perp}).$
 \end{lemma}
 \begin{proof}
 	For any ${\mathbf{X}_4}\in \Gamma(\mathfrak{D}), ~\mathbf{V_4}\in \Gamma(\mathfrak{D}^{\nu})$ and ${\mathbf{Z}_4}\in \Gamma(\mathfrak{D}^{\perp}),$ we have
 	\begin{equation*}
 	\begin{split}
 	g(\Sigma({\mathbf{X}_4},{\mathbf{Z}_4}),\mathbf{FV})&=g(\tilde{\nabla}_{{\mathbf{Z}_4}}{\mathbf{X}_4},\Phi \mathbf{V_4})+g(\tilde{\nabla}_{{\mathbf{Z}_4}}\mathbf{TV}, {\mathbf{X}_4}),\\
 	&=g((\tilde{\nabla}_{\mathbf{Z}_4}\Phi){\mathbf{X}_4},\mathbf{V_4})-g(\tilde{\nabla}_{\mathbf{Z}_4}\Phi {\mathbf{X}_4}, \mathbf{V_4})+\mathbf{TV}(\ln \mathbf{\phi}_{2})g({\mathbf{X}_4}, {\mathbf{Z}_4}).
 	\end{split}
 	\end{equation*}
 	Using \eqref{a3}, \eqref{b1} and the orthogonality of vector fields, we get $g(\Sigma({\mathbf{X}_4},{\mathbf{Z}_4}),\mathbf{FV})=0,$ which is second equality. On the other hand, we have
 	\begin{equation*}
 	\begin{split}
 	g(\Sigma({\mathbf{X}_4},{\mathbf{Z}_4}),\mathbf{FV})&=g(\tilde{\nabla}_{{\mathbf{X}_4}}{\mathbf{Z}_4},\Phi \mathbf{V_4})+g(\tilde{\nabla}_{{\mathbf{X}_4}}\mathbf{TV},{\mathbf{Z}_4}),\\
 	&=g((\tilde{\nabla}_{\mathbf{X}_4}\Phi){\mathbf{Z}_4},\mathbf{V_4})-g(\tilde{\nabla}_{\mathbf{X}_4}\Phi {\mathbf{Z}_4}, \mathbf{V_4})+\mathbf{TV}(\ln \mathbf{\phi}_{1})g({\mathbf{X}_4}, {\mathbf{Z}_4}).
 	\end{split}
 	\end{equation*}
 	Again use of \eqref{a3}, \eqref{a5} and the orthogonality of vector fields, we get
 	$$g(\Sigma({\mathbf{X}_4},\mathbf{V_4}),\Phi {\mathbf{Z}_4})=g(\Sigma({\mathbf{X}_4},{\mathbf{Z}_4}),\mathbf{FV}).$$
 	Hence the claim.
 \end{proof}
 \section{Chen type inequality for Bi-Warped product immersions}
 In this section, we show how the squared norm of the second fundamental form and the bi-warped product's warping functions are related. We will conduct a double analysis.
\\
Specifically, we will show a Chen type inequality for Bi-Warped product immersions, in two different settings (\textit{Theorem 4.1} and \textit{Theorem 4.2}).
In this regard, let's start by considering the following orthogonal frame field to show our primary thesis and to be able to provide this first relationship.\\
 
 Let $\mathbf{L}=\mathbf{L}_{\nu}\times_{\mathbf{\phi}_{1}}\mathbf{L}_{\mathbf{T}}\times_{\mathbf{\phi}_{2}}\mathbf{L}_{\perp}$ be an $n$-dimensional bi-warped product submanifolds of a $(2n+1)$-dimensional locally conformal almost cosymplectic manifold $\tilde{\mathbf{L}}$ such that $\xi$ is tangent to the base manifold $\mathbf{L}_{\nu}.$ If the dimensions $\dim(\mathbf{L}_{\mathbf{T}})=n_{1}, \dim(\mathbf{L}_{\perp})=n_{2}$ and $\dim(\mathbf{L}_{\nu})=n_{3},$ then the orthogonal frames of the corresponding tangent spaces $\mathfrak{D}, \mathfrak{D}^{\perp}$ and $\mathfrak{D}^{\nu}$, respectively, are given by $\{e_{1},\dots,e_{p},e_{p+1}=\Phi e_{1},\dots,e_{n_{1}}=e_{2_{p}}=\Phi e_{p}\},~ \{e_{n_{1}+1}=\bar{e}_{1},\dots,e_{n_{1}+n_{2}}=\bar{e}_{n_{2}}\}$ and $\{e_{n_{1}+n_{2}+1}=e^{*}_{1},\dots,e_{n_{1}+n_{2}+q}=e^{*}_{q},e_{n_{1}+n_{2}+q+1}=e^{*}_{q+1}=\sec\nu \mathbf{T}e^{*}_{1},\dots,e_{n_{1}+n_{2}+2q}=e^{*}_{2q}=\sec\nu \mathbf{T}e^{*}_{q},e_{m}=e^{*}_{n_{3}}=e^{*}_{2q+1}=\xi \}.$ Then the orthonormal frame fields of the normal subbundles of $ \Phi \mathfrak{D}^{\perp}, F\mathfrak{D}^{\nu}$ and $\mu$, respectively, are $\{e_{m+1}=\tilde{e}_{1}=\Phi \bar{e}_{1},\dots,e_{m+n_{2}}=\tilde{e}_{n_{2}}=\Phi \bar{e}_{n_{2}}\},$ $\{e_{m+n_{2}+1}=\tilde{e}_{n_{2}+1}=\csc\nu \mathbf{\phi}e_{1}^{*},\dots,e_{m+n_{2}+q}=\tilde{e}_{m_{2}+q}=\csc\nu \mathbf{\phi}e^{*}_{q}, e_{m+n_{2}+q+1}=\tilde{e}_{n_{2}+q+1}=\csc\nu \sec\nu \mathbf{FT}e^{*}_{1},\dots,e_{m+n_{2}+n_{3}-1}=\tilde{e}_{n_{2}+n_{3}-1}=\csc\nu \sec\nu \mathbf{FT}e^{*}_{q}\}$ and $\{e_{m+n_{2}+n_{3}}=\tilde{e}_{n_{2}+n{3}},\dots,e_{2n+1}=\tilde{e}_{2_{(n-n_{2}-n_{3}+1)-n_{1}}}\}.$
 \begin{theorem}
 	Let $\mathbf{L}=\mathbf{L}_{\nu}\times_{\mathbf{\phi}_{1}}\mathbf{L}_{\mathbf{T}}\times_{\mathbf{\phi}_{2}}\mathbf{L}_{\perp}$ be a $\mathfrak{D}^{\perp}-\mathfrak{D}^{\nu}$ mixed totally geodesic bi-warped product submanifold of a locally conformal almost cosymplectic manifold $\tilde{\mathbf{L}}$ such that $\xi$ is tangent to $\mathbf{L}_{\nu}.$ Then
 	\begin{itemize}
 		\item [(i)] The squared norm of the second fundamental form $\Sigma$ of $\mathbf{L}$ satisfies
 		\begin{align}
 		||\Sigma||^{2}\geq&n_{1}\Big(1+\cot^2\nu\Big)\Big(1+\cos^{2}\nu\Big)\bigg(||\vec{\nabla}(\ln \mathbf{\phi}_{1})||^{2}-\beta^2\bigg)\notag\\
 		&+n_{2}\Big(\csc^2\nu-1\Big)\bigg(||\vec{\nabla}(\ln \mathbf{\phi}_{2})||^{2}-\beta^2\bigg)
 		\end{align}
 		where $n_{1}=\dim(\mathbf{L}_{\mathbf{T}}), n_{2}=\dim(\mathbf{L}_{\perp})$ and $\vec{\nabla}(\ln \mathbf{\phi}_{1})$ is the gradient of $\ln \mathbf{\phi}_{1}$ along $\mathbf{L}_{\mathbf{T}}$ and $\vec{\nabla}(\ln \mathbf{\phi}_{2})$ is the gradient of $\ln \mathbf{\phi}_{2}$ along $\mathbf{L}_{\perp}.$\\
 		\item [(ii)]  if the equality sign in $(i)$ holds identically, then $\mathbf{L}_ {\nu}$ is totally geodesic submanifold of $\tilde{\mathbf{L}}, \mathbf{L}_{\mathbf{T}}$  and $\mathbf{L}_\perp$ are totally umbilical submanifolds of $\tilde{\mathbf{L}}$. In addition, $\mathbf{L}$ is a completely $\mathfrak{D}^{\nu}$-geodesic submanifold of $\tilde{\mathbf{L}}$.
 	\end{itemize}
 \end{theorem}
 \begin{proof}
 	From \eqref{a7} and \eqref{a8}, we get
 	$$||\Sigma||^{2}=\sum_{i,j=1}^{m}g(\Sigma(e_{i},e_{j}),\Sigma(e_{i},e_{j}))=\sum_{r=m+1}^{2n+1}\sum_{i,j=1}^{m}g(\Sigma(e_{i},e_{j}),e_{r})^{2}.$$
 	The above relation can take the form:
 	\begin{align}\label{c1}
 	\begin{split}
 	||\Sigma||^{2}&=\sum_{r=1}^{n_{2}}\left(\sum_{i,j=1}^{m}g(\Sigma(e_{i},e_{j}),\tilde{e}_{r})^{2}\right)\\
 	&+\sum_{r=n_{2}+1}^{n_{2}+n_{3}-1}\left(\sum_{i,j=1}^{m}g(\Sigma(e_{i},e_{j}),\tilde{e}_{r})^{2}\right)\\
 	&+\sum_{r=n_{2}+n_{3}}^{2(n-n_{2}-n_{3}+1)-n_{1}}\left(\sum_{i,j=1}^{m}g(\Sigma(e_{i},e_{j}),\tilde{e}_{r})^{2}\right).
 	\end{split}
 	\end{align}

We derive $\mathfrak{D}, \mathfrak{D}^{\perp}$ and $\mathfrak{D}^{\nu},$ by using the constructed frame fields of $\mathfrak{D}, \mathfrak{D}^{\perp}$ and $\mathfrak{D}^{\nu},$, leaving the third $\mu$-components positive terms because we could not find any connection for bi-warped products in terms of the $\mu$-components.
 	\begin{align}\label{c2}
 	\begin{split}
 	||\Sigma||^{2} &\geq \sum_{r=1}^{n_{2}}\sum_{i,j=1}^{n_{1}}g(\Sigma(e_{i},e_{j}),\Phi \bar{e}_{r})^{2}+\sum_{r=1}^{n_{2}}\sum_{i,j=1}^{n_{2}}g(\Sigma(\bar{e}_{i},\bar{e}_{j}),\Phi \bar{e}_{r})^{2}\\
 	&+\sum_{r=1}^{n_{2}}\sum_{i,j=1}^{n_{3}}g(\Sigma(e_{i}^{*},e_{j}^{*}),\Phi \bar{e}_{r})^{2}+2\sum_{r=1}^{n_{2}}\sum_{i=1}^{n_{1}}\sum_{j=1}^{n_{2}}g(\Sigma(e_{i},\bar{e}_{j}),\Phi \bar{e}_{r})^{2}\\
 	&+2\sum_{r=1}^{n_{2}}\sum_{i=1}^{n_{1}}\sum_{j=1}^{n_{3}}g(h(e_{i},e_{j}^{*}),\Phi \bar{e}_{r})^{2}+2\sum_{r=1}^{n_{2}}\sum_{i=1}^{n_{1}}\sum_{j=1}^{n_{3}}g(\Sigma(\bar{e}_{i},e_{j}^{*}),\Phi \bar{e}_{r})^{2}\\
 	&+\sum_{r=n_{2}+1}^{n_{2}+n_{3}-1}\sum_{i,j=1}^{n_{1}}g(\Sigma(e_{i},e_{j}),\Phi \tilde
 	{e}_{r})^{2}+\sum_{r=n_{2}+1}^{n_{2}+n_{3}-1}\sum_{i,j=1}^{n_{2}}g(\Sigma(\bar{e}_{i},\bar{e}_{j}),\Phi \tilde{e}_{r})^{2}\\
 	&+\sum_{r=n_{2}+1}^{n_{2}+n_{3}-1}\sum_{i,j=1}^{n_{3}}g(h(e_{i}^{*},e_{j}^{*}),\Phi \tilde{e}_{r})^{2}+2\sum_{r=n_{2}+1}^{n_{2}+n_{3}-1}\sum_{i=1}^{n_{1}}\sum_{j=1}^{n_{2}}g(\Sigma(e_{i},\bar{e}_{j}),\Phi \tilde{e}_{r})^{2}\\
 	&+2\sum_{r=n_{2}+1}^{n_{2}+n_{3}-1}\sum_{i=1}^{n_{1}}\sum_{j=1}^{n_{3}}g(e_{i}^{*},e_{j}^{*}),\Phi \tilde{e}_{r})^{2}+2\sum_{r=n_{2}+1}^{n_{2}+n_{3}-1}\sum_{i=1}^{n_{2}}\sum_{j=1}^{n_{3}}g(\Sigma(\bar{e}_{i},e_{j}^{*}),\Phi \tilde{e}_{r})^{2}.
 	\end{split}
 	\end{align}
 	Using Lemma $3.1 (iii)$ with the mixed completely geodesic condition $\mathfrak{D}^{\perp}-\mathfrak{D}^{\nu}$, the third term in the right-hand side is identically zero. Similarly, the fifth and tenth terms vanish when Lemma $3.2$ is used, whereas the sixth and twelfth terms are zero when the $\mathfrak{D}^{\perp}-\mathfrak{D}^{\nu}$ mixed totally condition is used. Using Lemma $3.1$, the eleventh term is also zero (iv).  However, we were unable to identify a relationship for the first, second, fourth, and ninth terms in the right-hand side of the above inequality for these bi-warped products, so we will leave them positive. After that, just the seventh and eighth terms must be assessed, and the statement above can be expressed as
 
 	\begin{equation*}
 	\begin{split}
 	||\Sigma||^{2}&\geq \sum_{r=1}^{q}\sum_{i,j=1}^{n_{1}}g(\Sigma(e_{i},e_{j}),\csc\nu Fe_{r}^{*})^{2}+\sum_{r=1}^{q}\sum_{i,j=1}^{n_{1}}g(\Sigma(e_{i},e_{j}),\csc\nu \sec\nu  \mathbf{FT}e_{r}^{*})^{2}\\
 	&+\sum_{r=1}^{q}\sum_{i,j=1}^{n_{2}}g(\Sigma(\bar{e}_{i},\bar{e}_{j}),\csc\nu \mathbf{\phi}e_{r}^{*})^{2}+\sum_{r=1}^{q}\sum_{i,j=1}^{n_{2}}g(\Sigma(\bar{e}_{i},\bar{e}_{j}),\csc\nu \sec\nu  \mathbf{FT}e_{r}^{*})^{2}.
 	\end{split}
 	\end{equation*}
 	Using Lemma $3.1(i)$ and the relations \eqref{b5}-\eqref{b11} in first two terms and using Lemma $3.1(ii)$ and \eqref{b12} in last two terms, we get
 	\begin{equation*}
 	\begin{split}
 	||\Sigma||^{2} &\geq n_{1}\csc^{2}\nu (1+\sec^{2}\nu)\sum_{r=1}^{q}(\mathbf{T}e_{r}^{*}(\ln\mathbf{\phi}_{1}))^{2}+\beta n_{1}\csc^{2}\nu (1+\cos^{2}\nu)\sum_{r=1}^{q}(e_{r}^{*}(\ln\mathbf{\phi}_{1})-\eta(e_{r}^{*}))^{2}\\
 	&+ n_{2}\csc^{2}\nu \sum_{r=1}^{q}(\mathbf{T}e_{r}^{*}(\ln\mathbf{\phi}_{2}))^{2}+\beta n_{1}\csc^{2}\nu \cos^{2}\nu\sum_{r=1}^{q}(e_{r}^{*}(\ln\mathbf{\phi}_{2})-\eta(e_{r}^{*}))^{2}\\
 	&=n_{1}\csc^{2}\nu (1+\sec^{2}\nu)\sum_{r=1}^{2q+1}(\mathbf{T}e_{r}^{*}(\ln\mathbf{\phi}_{1}))^{2}- n_{1}\csc^{2}\nu (1+\sec^{2}\nu)\sum_{r=q+1}^{2q}g(e_{r}^{*},\mathbf{T}\vec{\nabla}(ln\mathbf{\phi}_{1}))^{2}\\
 	&-n_{1}\csc^{2}\nu (1+\sec^{2}\nu)(\mathbf{T}e_{2q+1}^{*}(\ln\mathbf{\phi}_{1}))^{2}+\beta n_{1}\csc^{2}\nu (1+\cos^{2}\nu)\sum_{r=1}^{q}(e_{r}^{*}(\ln\mathbf{\phi}_{1})^{2}\\
 	&+n_{2}\csc^{2}\nu \sum_{r=1}^{2q+1}(\mathbf{T}e^{*}_{r}(\ln\mathbf{\phi}_{2}))^{2}-n_{2}\csc^{2}\nu \sum_{r=q+1}^{2q}g(e^{*}_{r},\mathbf{T}\vec{\nabla}(\ln\mathbf{\phi}_{2}))^{2}\\
 	&-n_{2}\csc^{2}\nu (\mathbf{T}e^{*}_{2q+1}(ln\mathbf{\phi}_{2}))^{2}+\beta n_{2}\csc^{2}\nu \cos^{2}\nu \sum_{r=1}^{q}(e^{*}_{r}(\ln\mathbf{\phi}_{2}))^{2}.\\
 	\end{split}
 	\end{equation*}
 	The third and seventh terms in the right-hand side of the last relation vanish identically because $e^{*}_{2q+1}=\xi$ and $\mathbf{T}\xi=0$. As a result, using \eqref{a9}, the aforementioned inequality has the form
 	
 	\begin{equation*}
 	\begin{split}
 	||\Sigma||^{2}&\geq n_{1}\csc^{2}\nu (1+\sec^{2}\nu)||\mathbf{T}\vec{\nabla}(\ln\mathbf{\phi}_{1})||^{2}-n_{1}\csc^{2}\nu (1+\sec^{2}\nu)\sec^{2}\nu \sum_{r=1}^{q}g(\mathbf{T}e^{*}_{r},\mathbf{T}\vec{\nabla}(\ln \mathbf{\phi}_{1}))^{2}\\
 	&+\beta n_{1}\csc^{2}\nu(1+\cos^{2}\nu)\sum_{r=1}^{q}(e^{*}_{r}(\ln\mathbf{\phi}_{1}))^{2}+n_{2}\csc^{2}\nu||\mathbf{T}\vec{\nabla}(\ln\mathbf{\phi}_{2})||^{2}+ \\ &-n_{2}\csc^{2}\nu \sec^{2}\nu \sum_{r=1}^{q}g(\mathbf{T}e_{r}^{*},\mathbf{T}\vec{\nabla}(\ln\mathbf{\phi}_{2}))^{2}+\beta n_{2}\csc^{2}\nu \cos^{2}\nu \sum_{r=1}^{q}(e_{r}^{*}(\ln\mathbf{\phi}_{2}))^{2}.
 	\end{split}
 	\end{equation*}
 	Using \eqref{a9}, \eqref{a12} and the fact that $\xi(\ln f_{i})=\beta,$ $i=1,2$ (\textit{Theorem 3.2}), we get
 	$$||\Sigma||^{2}\geq \beta n_{1}\csc^{2}\nu (1+\cos^{2}\nu)(||\vec{\nabla}(\ln\mathbf{\phi}_{1})||^{2}-1)+\beta n_{2}\cot^{2}\nu(||\vec{\nabla}(\ln\mathbf{\phi}_{2})||^{2}-1),$$
 	which is inequality $(i).$ \\
 	
 	Now, we discuss the equality case, from the leaving third $\mu$-components terms in the right-hand side of \eqref{c1}, we have
 	\begin{align}\label{c3}
 	\Sigma(TL,TL)\perp \mu. 
 	\end{align}
 	We may deduce from the first, second, and fourth terms in \eqref{c2} that
 	\begin{align}\label{c4}
 	\Sigma(\mathfrak{D},\mathfrak{D})\perp \Phi \mathfrak{D}^{\perp},\notag\\ \Sigma(\mathfrak{D}^{\perp},\mathfrak{D}^{\perp})\perp \Phi \mathfrak{D}^{\perp},\notag\\ \Sigma(\mathfrak{D},\mathfrak{D}^{\perp})\perp \Phi \mathfrak{D}^{\perp}.
 	\end{align}
 	The leaving ninth term of \eqref{c2}, we get 
 	\begin{align}\label{c5}
 	\Sigma(\mathfrak{D}^{\nu},\mathfrak{D}^{\nu})\perp F\mathfrak{D}^{\nu}.
 	\end{align}
 	Then \eqref{c3} and \eqref{c4}, give the following
 	\begin{align}\label{c6}
 	\Sigma(\mathfrak{D},\mathfrak{D})\in \mathbf{\phi} \mathfrak{D}^{\nu},~ \Sigma(\mathfrak{D}^{\perp},\mathfrak{D}^{\perp})\in \mathbf{\phi} \mathfrak{D}^{\nu},~ \Sigma(\mathfrak{D},\mathfrak{D}^{\perp})\in \mathbf{\phi}\mathfrak{D}^{\nu}.
 	\end{align}
 	As $\mathbf{L}$ is $\mathfrak{D}^{\perp}-\mathfrak{D}^{\nu}$ mixed totally geodesic, we conclude
 	\begin{align}\label{c7}
 	\Sigma(\mathfrak{D}^{\perp},\mathfrak{D}^{\nu})=0.
 	\end{align}
 	From the deleting third term of \eqref{c2}, we arrive
 	\begin{align}\label{c8}
 	\Sigma(\mathfrak{D}^{\nu},\mathfrak{D}^{\nu})\perp \Phi \mathfrak{D}^{\perp}.
 	\end{align}
 	From \eqref{c3}, \eqref{c5} and \eqref{c8}, we get 
 	\begin{align}\label{c9}
 	\Sigma(\mathfrak{D}^{\nu},\mathfrak{D}^{\nu})=0.
 	\end{align}
 	Similarly, from the deleting fifth, tenth and eleventh terms of \eqref{c2}, we find, respectively,
 	\begin{align}\label{c10}
 	\Sigma(\mathfrak{D},\mathfrak{D}^{\nu})\perp \Phi \mathfrak{D}^{\perp},~ \Sigma(\mathfrak{D},\mathfrak{D}^{\perp})\perp \mathbf{\phi}\mathfrak{D}^{\nu},~ \Sigma(\mathfrak{D},\mathfrak{D}^{\nu})\perp \mathbf{\phi} \mathfrak{D}^{\nu}.
 	\end{align}
 	Thus, \eqref{c3}, \eqref{c4} and \eqref{c10}, we obtain
 	\begin{align}\label{c11}
 	\Sigma(\mathfrak{D},\mathfrak{D}^{\nu})=0,~~\Sigma(\mathfrak{D},\mathfrak{D}^{\perp})=0.
 	\end{align}
 	Since $\mathbf{L}_{\mathbf{T}}$ is a totally geodesic submanifold of $\mathbf{L}.$ Combining with \eqref{c7}, \eqref{c9} and \eqref{c11}, we reached that $\mathbf{L}_{\nu}$ is totally geodesic in $\tilde{\mathbf{L}}$. \\
 	
 	On the other hand, since $\mathbf{L}_{\mathbf{T}}$ and $\mathbf{L}_{\perp}$ are totally umbilical, using \eqref{c6}, we conclude that $\mathbf{L}_{\mathbf{T}}$ and $\mathbf{L}{\perp}$ are totally umbilical submanifolds of $\tilde{\mathbf{L}}$, which yields to $(ii).$ \\
 	
 	Moreover, all conditions including \eqref{c4}-\eqref{c11} imply that $\mathbf{L}$ is a $\mathfrak{D}^{\nu}$-totally geodesic submanifold of $\tilde{\mathbf{L}}$.
 \end{proof}

 \begin{lemma}
 	Let $\mathbf{L}=\mathbf{L}_{\mathbf{T}}\times_{\mathbf{\phi}_{1}} \mathbf{L}_{\perp} \times _{\mathbf{\phi}_{2}} \mathbf{L}_{\nu}$ be a bi-warped product submanifold of locally conformal almost cosymplectic manifold $\tilde{\mathbf{L}}$ such that $\xi \in (\mathfrak{D}).$ Then, we have
 	\begin{itemize}
 		\item [(i)] $g(\Sigma({\mathbf{X}_4},{\mathbf{Y}_4}),\Phi \mathbf{V_4})=0;$
 		\item[(ii)] $g(\Sigma({\mathbf{X}_4},\mathbf{V_4}),\Phi \mathbf{U_4})=-\Phi {\mathbf{X}_4}(\ln\mathbf{\phi}_{1})g(\mathbf{U_4},\mathbf{V_4}).$
 	\end{itemize}
 	for any ${\mathbf{X}_4},{\mathbf{Y}_4}\in \mathfrak{D}$ and $\mathbf{U_4},\mathbf{V_4}\in \mathfrak{D}^{\perp}.$
 \end{lemma}
 \begin{proof}
 	The first component is straightforward to prove using Gauss-Weingarten formulae and the structure equation of a locally conformal almost cosymplectic manifold with \eqref{b1}. Now, we get
 	(ii) for any ${\mathbf{X}_4}\in\mathfrak{D}$ and $\mathbf{U_4},\mathbf{V_4} \in \mathfrak{D}^{\perp},$
 	\begin{equation*}
 	\begin{split}
 	g(\Sigma({{\mathbf{X}_4}},\mathbf{V_4}),\Phi \mathbf{U_4})&=g(\tilde{\nabla}_{\mathbf{V_4}}{\mathbf{X}_4},\Phi\mathbf{U_4})\\
 	&=g((\tilde{\nabla}_{\mathbf{V_4}}\Phi){\mathbf{X}_4},\mathbf{U_4})-g(\tilde{\nabla}_{\mathbf{V_4}}\Phi \mathbf{L}{\mathbf{X}_4},\mathbf{U_4}).
 	\end{split}
 	\end{equation*}
 	Now, using \eqref{a3} and \eqref{b1} and the perpendicularity of the vector fields, we get 
 	\begin{align*}
 	g(\Sigma({\mathbf{X}_4},\mathbf{V_4}),\Phi \mathbf{U_4})=-\Phi {\mathbf{X}_4}(\ln
 	\mathbf{\phi}_{1})g(\mathbf{U_4},\mathbf{L}V).
 	\end{align*}
 	Hence the claim.
 \end{proof}
 \begin{lemma}
 	Let $\mathbf{L}=\mathbf{L}_{\mathbf{T}}\times_{\mathbf{\phi}_{1}} \mathbf{L}_{\perp} \times _{\mathbf{\phi}_{2}} \mathbf{L}_{\nu}$ be a bi-warped product submanifold of locally conformal almost cosymplectic manifold $\tilde{\mathbf{L}}$ such that $\xi \in (\mathfrak{D}).$ Then, we have
 	\begin{itemize}
 		\item [(i)] $g(\Sigma({\mathbf{X}_4},{\mathbf{Y}_4}),\mathbf{FZ})=0;$
 		\item[(ii)] $g(\Sigma({\mathbf{X}_4},{\mathbf{Z}_4}),\mathbf{FW})=-\Phi {\mathbf{X}_4}(\ln\mathbf{\phi}_{2})g({\mathbf{Z}_4},\mathbf{W_4})+\Big\{{\mathbf{X}_4}(\ln\mathbf{\phi}_{2})-\beta \eta({\mathbf{X}_4})\Big\}g(\mathbf{TZ_4},\mathbf{W_4}).$
 	\end{itemize}
 	for any ${\mathbf{X}_4},{\mathbf{Y}_4}\in \mathfrak{D}$ and ${\mathbf{Z}_4},\mathbf{W_4}\in \mathfrak{D}^{\nu}.$
 \end{lemma}
 \begin{proof}
 	For any ${\mathbf{Y}_4}\in \mathfrak{D}$ and ${\mathbf{Z}_4}\in \mathfrak{D}^{\nu}.$
 	\begin{align}\label{d1}
 	g(\Sigma({\mathbf{X}_4},{\mathbf{Y}_4}),\mathbf{FZ})&=g(\tilde{\nabla}_{{\mathbf{X}_4}}{\mathbf{Y}_4},\mathbf{FZ})\notag\\
 	&=g(\tilde{\nabla}_{{\mathbf{X}_4}}{\mathbf{Y}_4},\Phi {\mathbf{Z}_4})-g(\tilde{\nabla}_{{\mathbf{X}_4}}{\mathbf{Y}_4},\mathbf{PZ}).
 	\end{align}
 	Since $g({\mathbf{Y}_4}, \mathbf{TZ_4})=g({\mathbf{Y}_4},\Phi {\mathbf{Z}_4})=-g(\Phi {\mathbf{Y}_4}, {\mathbf{Z}_4})=0,$
 	from \eqref{d1} and \eqref{a3}, we obtain 
 	\begin{align*}
 	g(\Sigma({\mathbf{X}_4},{\mathbf{Y}_4}), \mathbf{FZ})&=\beta (g(\Phi {\mathbf{Y}_4}, \tilde{\nabla}_{{\mathbf{X}_4}}{\mathbf{Z}_4})+g({\mathbf{Y}_4},\tilde{\nabla}_{{\mathbf{X}_4}}\mathbf{PZ}))\\
 	&=\beta(g(\Phi {\mathbf{Y}_4},\nabla_{{\mathbf{X}_4}}{\mathbf{Z}_4})+g({\mathbf{Y}_4},\nabla_{{\mathbf{X}_4}}\mathbf{PZ})).
 	\end{align*}
 	Hence using \eqref{b1}, we get the desired results. Again for any ${\mathbf{X}_4},{\mathbf{Y}_4}\in \mathfrak{D}$ and ${\mathbf{Z}_4},\mathbf{W_4}\in \mathfrak{D}^{\nu}.$
 	\begin{equation*}
 	\begin{split}
 	g(\Sigma({\mathbf{X}_4},{\mathbf{Z}_4}),\mathbf{FW})&=g(\tilde{\nabla}_{{\mathbf{Z}_4}}{\mathbf{X}_4},\Phi \mathbf{W_4})-g(\tilde{\nabla}_{{\mathbf{Z}_4}}{\mathbf{X}_4}, \mathbf{TW})\\
 	&=g(\Phi \tilde{\nabla}_{{\mathbf{Z}_4}}{\mathbf{X}_4},\mathbf{W_4})-g(\nabla_{{\mathbf{Z}_4}}{\mathbf{X}_4}, \mathbf{TW}).
 	\end{split}
 	\end{equation*}
 	By using the co-variant derivative property of $\Phi$ and \eqref{b1}, we get 
 	\begin{align*}
 	g(\Sigma({\mathbf{X}_4},{\mathbf{Z}_4}),\mathbf{FW})=g((\tilde{\nabla}_{{\mathbf{Z}_4}}\Phi){\mathbf{X}_4},{\mathbf{X}_4}\mathbf{W_4})-g(\tilde{\nabla}_{{\mathbf{Z}_4}}\Phi {\mathbf{X}_4},\mathbf{W_4})-{\mathbf{X}_4}(\ln
 	\mathbf{\phi}_{2})g({\mathbf{Z}_4}, \mathbf{TW}).
 	\end{align*}
 	Applying \eqref{a3} and \eqref{b1}, we obtained the required results.
 \end{proof}
 Interchanging the following relations yields the following results $X$ by $\Phi {\mathbf{X}_4}$, ${\mathbf{Z}_4}$ by $\mathbf{TZ_4}$ and $\mathbf{W_4}$ by $\mathbf{TW_4}$ in Lemma $(4.1),$ \\
 \begin{align}\label{d2}
 g(\Sigma(\Phi {\mathbf{X}_4},{\mathbf{Z}_4}),\mathbf{FW_4})&=\Big\{{\mathbf{X}_4}(\ln \mathbf{\phi}_{2})-\beta \eta({\mathbf{X}_4})\Big\}g({\mathbf{Z}_4},\mathbf{W_4})+\Phi {\mathbf{X}_4}(\ln \mathbf{\phi}_{2})g(\mathbf{TZ_4}, \mathbf{W_4})\\
 \label{d3}
 g(\Sigma({\mathbf{X}_4}, \mathbf{TZ_4}),\mathbf{FW_4})&=-\Phi {\mathbf{X}_4}(\ln \mathbf{\phi}_{2})g(\mathbf{TZ_4}, \mathbf{W_4})-\Big\{{\mathbf{X}_4}(\ln \mathbf{\phi}_{2})-\beta \eta ({\mathbf{X}_4})\Big\}\cos^{2}\nu g({\mathbf{Z}_4},\mathbf{W_4})\\
 \label{d4}
 g(\Sigma({\mathbf{X}_4},{\mathbf{Z}_4}),\mathbf{FTW_4})&=-\Phi {\mathbf{X}_4}(\ln \mathbf{\phi}_{2})g({\mathbf{Z}_4}, \mathbf{TW_4})-\Big\{{\mathbf{X}_4}(\ln \mathbf{\phi}_{2})-\beta \eta ({\mathbf{X}_4})\Big\}\cos^{2}\nu g({\mathbf{Z}_4},\mathbf{W_4})\\
 \label{d5}
 g(\Sigma(\Phi {\mathbf{X}_4},\mathbf{TZ_4}),\mathbf{FW_4})&=\Big\{{\mathbf{X}_4}(\ln \mathbf{\phi}_{2})-\beta \eta ({\mathbf{X}_4})\Big\}g(\mathbf{TZ_4},\mathbf{W_4})-\Phi {\mathbf{X}_4} (\ln \mathbf{\phi}_{2})\cos^{2}\nu g({\mathbf{Z}_4},\mathbf{W_4})\\
 \label{d6}
 g(\Sigma(\Phi {\mathbf{X}_4},{\mathbf{Z}_4}),\mathbf{FTW_4})&=\Big\{{\mathbf{X}_4}(\ln \mathbf{\phi}_{2})-\beta \eta ({\mathbf{X}_4})\Big\}g({\mathbf{Z}_4},\mathbf{TW_4})+\Phi {\mathbf{X}_4} (\ln \mathbf{\phi}_{2})\cos^{2}\nu g({\mathbf{Z}_4},\mathbf{W_4})\\
 \label{d7}
 g(\Sigma({\mathbf{X}_4},\mathbf{TZ_4}),\mathbf{FTW_4})&=-\Phi{\mathbf{X}_4}(\ln \mathbf{\phi}_{2})\cos^{2}\nu g({\mathbf{Z}_4},\mathbf{W_4})\notag\\
 &-\Big\{{\mathbf{X}_4}(\ln \mathbf{\phi}_{2})-\beta \eta ({\mathbf{X}_4})\Big\}\cos^{2}\nu g(\mathbf{TZ_4}, \mathbf{W_4})\\
 \label{d8}
 g(\Sigma(\Phi {\mathbf{X}_4}, \mathbf{TZ_4}),\mathbf{FTW_4})&=\Big\{{\mathbf{X}_4}(\ln \mathbf{\phi}_{2})-\beta \eta ({\mathbf{X}_4})\Big\}\cos^{2}\nu g({\mathbf{Z}_4}, \mathbf{W_4}).
 \end{align}
 \begin{lemma}
 	Let $\mathbf{L}=\mathbf{L}_{\mathbf{T}}\times_{\mathbf{\phi}_{1}} \mathbf{L}_{\perp}\times_{\mathbf{\phi}_{2}}\mathbf{L}_{\nu}$ be a bi-warped product submanifold of a locally conformal almost cosymplectic manifold $\tilde{\mathbf{L}}$, then, we have
 	\begin{itemize}
 		\item [(i)] $g(\Sigma({\mathbf{X}_4},{\mathbf{Z}_4}),\Phi \mathbf{V_4})=0,$
 		\item[(ii)] $g(\Sigma({\mathbf{X}_4},\mathbf{V_4}),\mathbf{FZ_4})=0.$
 	\end{itemize}
 	for any ${\mathbf{X}_4}\in \mathfrak{D}, ~\mathbf{V_4}\in \mathfrak{D}^{\perp}$ and ${\mathbf{Z}_4}\in \mathfrak{D}^{\nu}.$ Moreover, $\mathbf{L}_{\mathbf{T}}$ is an anti-invariant submanifold, $\mathbf{L}_{\perp}$ is an anti-invariant submanifold and $\mathbf{L}_{\nu}$ is a proper slant submanifold of $\tilde{\mathbf{L}}$. 
 \end{lemma}
 The following frame fields for a $m$-dimensional bi-warped product submanifold are now constructed. Let $\mathbf{L}=\mathbf{L}_{\mathbf{T}}\times_{\mathbf{\phi}_{1}}\mathbf{L}_{\perp}\times_{\mathbf{\phi}_{2}}\mathbf{L}_{\nu}$ be of a $(2n+1)$-dimension locally conformal almost cosymplectic manifold $\tilde{\mathbf{L}}$ such that $\xi$ is tangent to $\mathbf{L}_{\mathbf{T}}.$\\
 
 If the dimensions dim$(\mathbf{L}_{\mathbf{T}})=2_{t_{1}}+1,$ dim$(\mathbf{L}_{\perp})=t_{2},$ and dim$(\mathbf{L}_{\nu})=2t_{3},$ then the orthonormal frames of the corresponding tangent spaces are $\mathfrak{D}$, $\mathfrak{D}^{\perp}$ and $\mathfrak{D}^{\nu},$ respectively. Given by $\{e_{1},\cdots,e_{t_{1}},e_{t_{1}+1}=\Phi e_{1},\cdots,e_{2{t_{1}}}=\Phi e_{t_{1}}, e_{2_{t_{1}+1}}=\xi\},~ \{e_{2_{t_{1}+2}}=\bar{e}_{1},\cdots,e_{2_{t_{1}+t_{2}+1}}=\bar{e}_{t_{1}}\}$ and $\{e_{2_{t_{1}+t_{2}+2}}=e^{*}_{1},\cdots,e_{2_{t_{1}+1+t_{2}+t_{3}}}=e^{*}_{t_{3}},e_{2_{t_{1}+2+t_{2}+t_{3}}}=e^{*}_{t_{3}+1}=\sec\nu \mathbf{T}e^{*}_{1},\cdots,e_{m}=e^{*}_{2t_{3}}=\sec\nu \mathbf{T}e^{*}_{t_{3}},$ then the orthonormal frame fields of the normal subbundles of $ \Phi \mathfrak{D}^{\perp}, \mathbf{\phi}\mathfrak{D}^{\nu}$ and $\mu$, respectively, are $\{e_{m+1}=\tilde{e}_{1}=\Phi \bar{e}_{1},\cdots,e_{m+t_{2}}=\tilde{e}_{t_{2}}=\Phi \bar{e}_{t_{2}}\},$ $\{e_{m+t_{2}+1}=\tilde{e}_{t_{2}+1}=\csc\nu \mathbf{\phi}e_{1}^{*},\cdots,e_{m+t_{2}+t_{3}}=\tilde{e}_{t_{2}+t_{3}}=\csc\nu \mathbf{\phi}e^{*}_{t_{3}}, e_{m+t_{2}+t_{3}+1}=\tilde{e}_{t_{2}+t_{3}+1}=\cos\nu\sec\nu \mathbf{FT}e^{*}_{1},\cdots,e_{m+t_{2}+2t_{3}}=\tilde{e}_{t_{2}+2t_{3}}=\csc\nu\sec\nu \mathbf{FT}e^{*}_{t_{3}}\}$ and $\{e_{m+t_{2}+2t_{3}}=\tilde{e}_{t_{2}+2t_{3}+1},\cdots,e_{2n+1}=\tilde{e}_{2n+1-m-t_{2}-2t_{3}}\}.$
 \begin{theorem}
 	Let $\mathbf{L}=\mathbf{L}_{\mathbf{T}}\times_{\mathbf{\phi}_{1}}\mathbf{L}_{\perp}\times_{\mathbf{\phi}_{2}}\mathbf{L}_{\nu}$ be a $\mathfrak{D}^{\perp}-\mathfrak{D}^{\nu}$ mixed totally geodesic bi-warped product submanifold of a locally conformal almost cosymplectic manifold $\tilde{\mathbf{L}}$ such that $\xi$ is tangent to $\mathbf{L}_{\mathbf{T}}.$ Then
 	\begin{itemize}
 		\item [(i)] The squared norm of the second fundamental form $\Sigma$ of $\mathbf{L}$ satisfies
 		\begin{align}
 		||\Sigma||^{2}\geq&2t_{2}(||\vec{\nabla}(\ln \mathbf{\phi}_{1})||^{2}-\beta^{2})\notag\\
 		&+4t_{3}\csc^2\nu\Big(\sin^2\nu+2\cos^{2}\nu\Big)\Big(||\vec{\nabla}(\ln \mathbf{\phi}_{2})||^{2}-\beta^{2}\Big)
 		\end{align}
 		where $t_{2}=\dim(\mathbf{L}_{\perp}), 2t_{3}=\dim(\mathbf{L}_{\nu})$ and $\vec{\nabla}(\ln \mathbf{\phi}_{1})$ is the gradient of $\ln \mathbf{\phi}_{1}$ along $\mathbf{L}_{\perp}$ and $\vec{\nabla}(\ln \mathbf{\phi}_{2})$ is the gradient of $\ln \mathbf{\phi}_{2}$ along $\mathbf{L}_{\nu}.$\\
 		
 		\item [(ii)] If the equality sign in $(i)$ holds identically, then $\mathbf{L}_{\mathbf{T}}$ is a totally geodesic submanifold of $\tilde{\mathbf{L}}, \mathbf{L}_{\perp}$ and $\mathbf{L}_{\nu}$ are totally umbilical submanifolds of $\tilde{\mathbf{L}}$ with -$\vec{\nabla}(\ln \mathbf{\phi}f_{1})$ and -$\vec{\nabla}(\ln \mathbf{\phi}_{2}$) as mean curvature vectors, respectively.
 	\end{itemize}
 \end{theorem} 
 \begin{proof}
 	From the definition of $\Sigma$, we get
 	\begin{align*}
 	||\Sigma||^{2}=\sum_{i,j=1}^{m}g(\Sigma(e_{i},e_{j}),\Sigma(e_{i},e_{j}))=\sum_{r=m+1}^{2n+1}\sum_{i,j=1}^{m}g(\Sigma(e_{i},e_{j}),e_{r})^{2}.
 	\end{align*}
 	So the above expression can be rewritten in the form \\
 	\begin{equation}\label{d11}
 	\begin{split}
 	||\Sigma||^{2}&=\sum_{r=1}^{t_{2}}\left(\sum_{i,j=1}^{m}g(\Sigma(e_{i},e_{j}),\tilde{e}_{r})^{2}\right)\\
 	&+\sum_{r=t_{2}+1}^{t_{2}+2t_{3}}\left(\sum_{i,j=1}^{m}g(\Sigma(e_{i},e_{j}),\tilde{e}_{r})^{2}\right)\\
 	&+\sum_{r=t_{2}+2t_{3}+1}^{2n+1-t_{2}-2t_{3}-m}\left(\sum_{i,j=1}^{m}g(\Sigma(e_{i},e_{j}),\tilde{e}_{r})^{2}\right).
 	\end{split}
 	\end{equation}
 	Leaving the third $\mu$-components positive terms as we could not find any relation for bi-warped pproducts in terms of the $\mu$-components, then, by using the constructed frame fields of $\mathfrak{D}, \mathfrak{D}^{\perp}$ and $\mathfrak{D}^{\nu},$ we obtain 
 	\begin{equation}\label{d12}
 	\begin{split}
 	||\Sigma||^{2} &\geq \sum_{r=1}^{t_{2}}\sum_{i,j=1}^{2_{t_{1}}+1}g(\Sigma(e_{i},e_{j}),\Phi \bar{e}_{r})^{2}+\sum_{r=1}^{t_{2}}\sum_{i,j=1}^{t_{2}}g(\Sigma(\bar{e}_{i},\bar{e}_{j}),\Phi \bar{e}_{r})^{2}\\
 	&+\sum_{r=1}^{t_{2}}\sum_{i,j=1}^{2t_{3}}g(\Sigma(e_{i}^{*},e_{j}^{*}),\Phi \bar{e}_{r})^{2}+2\sum_{r=1}^{t_{2}}\sum_{i=1}^{2_{t_{1}}+1}\sum_{j=1}^{t_{2}}g(\Sigma(e_{i},\bar{e}_{j}),\Phi \bar{e}_{r})^{2}\\
 	&+2\sum_{r=1}^{t_{2}}\sum_{i=1}^{2_{t_{1}}+1}\sum_{j=1}^{2t_{3}}g(\Sigma(e_{i},e_{j}^{*}),\Phi \bar{e}_{r})^{2}+2\sum_{r=1}^{t_{2}}\sum_{i=1}^{t_{2}}\sum_{j=1}^{2t_{3}}g(\Sigma(\bar{e}_{i},e_{j}^{*}),\Phi \bar{e}_{r})^{2}\\
 	&+\sum_{r=1}^{2t_{3}}\sum_{i,j=1}^{2_{t_{1}+1}}g(\Sigma(e_{i},e_{j}),\Phi \tilde
 	{e}_{r})^{2}+\sum_{r=1}^{2t_{3}}\sum_{i,j=1}^{t_{2}}g(\Sigma(\bar{e}_{i},\bar{e}_{j}),\Phi \tilde{e}_{r})^{2}\\
 	&+\sum_{r=1}^{2t_{3}}\sum_{i,j=1}^{2t_{3}}g(\Sigma(e_{i}^{*},e_{j}^{*}),\Phi \tilde{e}_{r})^{2}+2\sum_{r=1}^{2t_{3}}\sum_{i=1}^{2_{t_{1}}+1}\sum_{j=1}^{t_{2}}g(\Sigma(e_{i},\bar{e}_{j}),\Phi \tilde{e}_{r})^{2}\\
 	&+2\sum_{r=1}^{2t_{3}}\sum_{i=1}^{2_{t_{1}}+1}\sum_{j=1}^{2t_{3}}g(e_{i},e_{j}^{*}),\Phi \tilde{e}_{r})^{2}+2\sum_{r=1}^{2t_{3}}\sum_{i=1}^{t_{2}}\sum_{j=1}^{2t_{3}}g(\Sigma(\bar{e}_{i},e_{j}^{*}),\Phi \tilde{e}_{r})^{2}.
 	\end{split}
 	\end{equation}
 	First and fifth terms in the right hand side of \eqref{d12} vanish identically by using Lemma $(4.1) (i)$ and Lemma $(4.3) (ii),$ respectively. Similarly, seventh and tenth terms are also identically zero by using Lemma $(4.2)(i)$ and Lemma $(4.3)(ii)$ respectively. In addition, there is no relation for $g(\Sigma(\bar{e}_{i},\bar{e}_{j}),\Phi \bar{e}_{r}),~i,j,r=1,...,t_{2},$ when the vectors are from the same space. We do not have any relations for bi-warped products of the following terms:\\
 	$g(\Sigma(e_{i}^{*},e_{j}^{*}),\Phi \bar{e}_{r}),~i,j=1,...,2t_{3}$ and $r=1,...,t_{2};$\\
 	$g(\Sigma,\Phi \bar{e}_{r}),~i,j=1,...,2t_{3}$ and $r=1,...,t_{2};$\\
 	$g(\Sigma(\bar{e}_{i}^{*},\bar{e}_{j}^{*}),\tilde{e}_{r}),~i,j=1,...,t_{2}$ and $r=1,...,2t_{3};$\\
 	$g(\Sigma(e_{i}^{*},e_{j}^{*}),\tilde{e}_{r}),~i,j,r=1,...,2t_{3};$\\
 	$g(\Sigma(\bar{e}_{i},e_{j}^{*}),\tilde{e}_{r}),~i=1,...,t_{2}$ and $j,r=1,...,2t_{3};$\\
 	Thus, we have left these positive terms but we will consider them for equality case. Then, we assume the forth and eleventh terms which are evaluated.\\
 	\begin{equation*}
 	\begin{split}
 	||\Sigma||^{2} &\geq 2\sum_{j,r=1}^{t_{2}}\sum_{i=1}^{2_{t_{1}}+1}g(\Sigma(e_{i},\bar{e}_{j}),\Phi \bar{e}_{r})^{2}+2\sum_{j,r=1}^{2t_{3}}\sum_{i=1}^{2_{t_{1}}+1}g(\Sigma(e_{i},e_{j}^{*}), \tilde{e}_{r})^{2}\\
 	&=2\sum_{j,r=1}^{t_{2}}\sum_{i=1}^{2_{t_{1}}}g(\Sigma(e_{i},\bar{e}_{j}),\Phi \bar{e}_{r})^{2}+2\sum_{j,r=1}^{t_{2}}g(\Sigma(e_{2_{t_{1}}+1},\bar{e}_{j}),\Phi \bar{e}_{r})^{2}\\
 	&+2\sum_{j,r=1}^{2t_{3}}\sum_{i=1}^{2_{t_{1}}}g(\Sigma(e_{i},e_{j}^{*}),\tilde{e}_{r})^{2}+2\sum_{j,r=1}^{2t_{3}}g(\Sigma(e_{2_{t_{1}}+1},e_{j}^{*}),\Phi \tilde{e}_{r})^{2}.
 	\end{split}
 	\end{equation*}
 	Since $e_{2_{t_{1}}+1}=\xi$ and for a locally conformal almost cosymplectic manifold $\Sigma(\xi ,{\mathbf{X}_4})=0,$ for any $X\in \mathbf{TM}$, then the second and forth terms in the right hand side of the above expression vanish identically. Thus, by use of frame fields of $\mathfrak{D}, \mathfrak{D}^{\nu}, \Phi \mathfrak{D}^{\perp}$ and $\mathbf{\phi}\mathfrak{D}^{\nu},$ we find\\
 	\begin{equation*}
 	\begin{split}
 	||\Sigma||^{2}&\geq 2\sum_{j,r=1}^{t_{2}}\sum_{i=1}^{t_{1}}g(\Sigma(e_{i},\bar{e}_{j}),\Phi \bar{e}_{r})^{2}+2\sum_{j,r=1}^{t_{2}}\sum_{i=1}^{t_{1}}g(\Sigma(\Phi e_{i},\bar{e}_{j}),\Phi \bar{e}_{r})^{2}\\
 	&+2\csc^{2}\nu \sum_{j,r=1}^{t_{3}}\sum_{i=1}^{t_{1}}g(\Sigma(e_{i},e^{*}_{j}),Fe^{*}_{r})^{2}+2\csc^{2}\nu \sum_{j,r=1}^{t_{3}}\sum_{i=1}^{t_{1}}g(\Sigma(\Phi e_{i},e^{*}_{j}),\mathbf{\phi}e^{*}_{r})^{2}\\
 	&+2\sec^{2}\nu \csc^{2}\nu \sum_{j,r=1}^{t_{3}}\sum_{i=1}^{t_{1}}g(\Sigma(e_{i},\mathbf{T}e^{*}_{j}),\mathbf{\phi}e^{*}_{r})^{2}+2sec^{2}\nu \csc^{2}\nu \sum_{j,r=1}^{t_{3}}\sum_{i=1}^{t_{1}}g(\Sigma(\Phi e_{i},\mathbf{T}e^{*}_{j}),\mathbf{\phi}e^{*}_{r})^{2}\\
 	&+2\sec^{2}\nu \csc^{2}\nu \sum_{j,r=1}^{t_{3}}\sum_{i=1}^{t_{1}}g(\Sigma(e_{i},e^{*}_{j}),\mathbf{FT}e^{*}_{r})^{2}+2sec^{2}\nu \csc^{2}\nu \sum_{j,r=1}^{t_{3}}\sum_{i=1}^{t_{1}}g(\Sigma(\Phi e_{i},e^{*}_{j}),\mathbf{FT}e^{*}_{r})^{2}\\
 	&+2\sec^{4}\nu \csc^{2}\nu \sum_{j,r=1}^{t_{3}}\sum_{i=1}^{t_{1}}g(\Sigma(e_{i},\mathbf{T}e^{*}_{j}),\mathbf{FT}e^{*}_{r})^{2}.
 	\end{split}
 	\end{equation*}
 	Then from Lemma $4.1(ii)$, Lemma $4.2(ii)$ and the relations \eqref{d2}-\eqref{d8}, the above expression takes the form
 	\begin{equation*}
 	\begin{split}
 	||\Sigma||^{2}&\geq 2t_{2}\sum_{i=1}^{t_{1}}(e_{i}(\ln\mathbf{\phi}_{1}))^{2}+2t_{2}\sum_{i=1}^{t_{1}}(\Phi e_{i}(\ln\mathbf{\phi}_{1}))^{2}\\
 	&+4t_{3}\csc^{2}\nu \sum_{i=1}^{t_{1}}(e_{i}(\ln\mathbf{\phi}_{2}))^{2}+4t_{3}\cot^{2}\nu \sum_{i=1}^{t_{1}}(e_{i}(ln\mathbf{\phi}_{2}))^{2}\\
 	&+4t_{3}\csc^{2}\nu \sum_{i=1}^{t_{1}}(\Phi e_{i}(\ln\mathbf{\phi}_{2}))^{2}+4t_{3}\cot^{2}\nu \sum_{i=1}^{t_{1}}(\Phi e_{i}(\ln\mathbf{\phi}_{2}))^{2}\\
 	&=2t_{2}\sum_{i=1}^{2_{t_{1}}+1}(e_{i}(\ln\mathbf{\phi}_{1}))^{2}-2t_{2}(e_{2_{t_{1}}+1}(\ln\mathbf{\phi}_{1}))^{2}\\
 	&+4t_{3}(1+2\cot^{2}\nu)\sum_{i=1}^{2_{t_{1}}+1}(e_{i}(\ln\mathbf{\phi}_{2}))^{2}-4t_{3}(1+2\cot^{2}\nu)(e_{2_{t_{1}}+1}(\ln\mathbf{\phi}_{2}))^{2}\\
 	\end{split}
 	\end{equation*}
 	Since $2{t_{1}}+1=\xi$, by $\xi(\ln\mathbf{\phi}_{1})=\beta$, $\xi(\ln\mathbf{\phi}_{2})=\beta$ and using
 	\eqref{a9}, we obtain
 	$$||\Sigma||^{2}\geq 2t_{2}(||\vec{\nabla}(\ln\mathbf{\phi}_{1})||^{2}-\beta^{2})+4t_{3}(1+2\cot^{2}\nu)(||\vec{\nabla}(\ln\mathbf{\phi}_{2}||)^{2}-4t_{3}(1+2\cot^{2}\nu)\beta^{2},$$
 	which is the inequality $(i).$\\
 	For the equality case, from the leaving third term in the right hand side of equation \eqref{d11}, we get 
 	\begin{equation}\label{d14}
 	\Sigma({\mathbf{X}_4},{\mathbf{Y}_4})\perp \mu 
 	\end{equation}
 	for any ${\mathbf{X}_4},{\mathbf{Y}_4}\in \mathbf{TM}.$\\
 	In addition from the vanishing first and seventh terms of \eqref{d12}, we find
 	\begin{equation}\label{d15}
 	\Sigma(\mathfrak{D},\mathfrak{D})\perp \Phi \mathfrak{D}^{\perp}, ~\Sigma(\mathfrak{D},\mathfrak{D})\perp \mathbf{\phi}\mathfrak{D}^{\nu}.
 	\end{equation}
 	Then from \eqref{d14} and \eqref{d15}, we get 
 	\begin{equation}\label{d16}
 	\Sigma(\mathfrak{D},\mathfrak{D})=\{0\}.
 	\end{equation}
 	Similarly, from the leaving second and eighth terms of \eqref{d12}, we find that
 	\begin{equation}\label{d17}
 	\Sigma(\mathfrak{D}^{\perp},\mathfrak{D}^{\perp})\perp \Phi \mathfrak{D}^{\perp}, ~\Sigma(\mathfrak{D}^{\perp},\mathfrak{D}^{\perp})\perp \mathbf{\phi}\mathfrak{D}^{\nu}.
 	\end{equation}
 	Thus, from \eqref{d14} and \eqref{d17}, we have 
 	\begin{equation}\label{d18}
 	\Sigma(\mathfrak{D}^{\perp},\mathfrak{D}^{\perp})=\{0\}.
 	\end{equation}
 	Now, from the leaving third and ninth terms of \eqref{d12}, we find that
 	\begin{equation}\label{d19}
 	\Sigma(\mathfrak{D}^{\nu},\mathfrak{D}^{\nu})\perp \Phi \mathfrak{D}^{\perp}, ~\Sigma(\mathfrak{D}^{\nu},\mathfrak{D}^{\nu})\perp \mathbf{\phi}\mathfrak{D}^{\nu}.
 	\end{equation}
 	Thus, from \eqref{d14} and \eqref{d18}, we have 
 	\begin{equation}\label{d20}
 	\Sigma(\mathfrak{D}^{\nu},\mathfrak{D}^{\nu})=\{0\}.
 	\end{equation}
 	Since $\mathbf{M_{T}}$ is totally geodesic submanifold of $\mathbf{L}$, by using \eqref{d16}, \eqref{d18} and \eqref{d20}, we conclude that $\mathbf{M_{T}}$ is totally geodesic in $\tilde{\mathbf{L}}$, which is the first part of the equality case.\\
 	On the other hand, the vanishing tenth term of \eqref{d12} with \eqref{d14}, given by
 	\begin{equation}\label{d21}
 	\Sigma(\mathfrak{D},\mathfrak{D}^{\perp})\subset \Phi \mathfrak{D}^{\perp}
 	\end{equation} 
 	Similarly, from the vanishing fifth term in \eqref{d12} with \eqref{d14}, we get
 	\begin{equation}\label{d22}
 	\Sigma(\mathfrak{D},\mathfrak{D}^{\nu})\subset \Phi \mathfrak{D}^{\nu}
 	\end{equation}
 	In addition, from the leaving sixth and twelfth terms in \eqref{d12}, we get  
 	\begin{equation}\label{d23}
 	\Sigma(\mathfrak{D}^{\perp},\mathfrak{D}^{\nu})\perp \Phi \mathfrak{D}^{\perp}, ~\Sigma(\mathfrak{D}^{\perp},\mathfrak{D}^{\nu})\perp \mathbf{\phi}\mathfrak{D}^{\nu}.
 	\end{equation}
 	Thus, from \eqref{d14} and \eqref{d23}, we have 
 	\begin{equation}\label{d24}
 	\Sigma(\mathfrak{D}^{\perp},\mathfrak{D}^{\nu})=\{0\}.
 	\end{equation}
 	On the other hand, from \eqref{c1}, we know that, for any $1\leq i\neq j\leq 2,$ and any vector field ${\mathbf{Z}_4}_{i}$ in $\mathbf{D}_{i}$ and ${\mathbf{Z}_4}_{j}$ in $\mathbf{D}_{j},$ we have 
 	$$\nabla_{{\mathbf{Z}_4}_{i}}{\mathbf{Z}_4}_{j}=0,$$ 
 	which implies \\
 	$$g(\nabla_{{\mathbf{Z}_4}_{i}}\mathbf{W_4}_{i},{\mathbf{Z}_4}_{j})=0.$$
 	Using this fact, if $h^{\perp}$ denotes the second fundamental form of $\mathbf{L}_{\perp}$ in $\mathbf{L}$, then we have \\
 	$g(\Sigma^{\perp}(\mathbf{U_4},\mathbf{V_4}),{\mathbf{X}_4})=g(\nabla_{\mathbf{U_4}}\mathbf{V_4},{\mathbf{X}_4})=g(\tilde{\nabla}_{\mathbf{U_4}}\mathbf{V_4},{\mathbf{X}_4})=-g(\tilde{\nabla}_{\mathbf{U_4}}\mathbf{X,V}).$\\
 	for any $\mathbf{U,V}\in \mathfrak{D}^{\perp}$ and ${\mathbf{X}_4}\in \mathfrak{D}$. Using \eqref{a4} and \eqref{b1}, we find \\
 	$$g(\Sigma^{\perp}(\mathbf{U,V),X})=-{\mathbf{X}_4}(\ln\mathbf{\phi}_{1})g(\mathbf{U,V}).$$
 	or equivalently,
 	\begin{equation}\label{d25}
 	\Sigma^{\perp}(\mathbf{U,V})=-\vec{\nabla}(\ln\mathbf{\phi}_{1})g(\mathbf{U,V}).
 	\end{equation}
 	Similarly, if $\Sigma^{\nu}$ is the second fundamental form of $\mathbf{L}_{\nu}$ in $\mathbf{L},$ then we can obtain 
 	\begin{equation}\label{d26}
 	\Sigma^{\nu}(\mathbf{Z,W})=-\vec{\nabla}(\ln\mathbf{\phi}_{2})g(\mathbf{Z,W}),
 	\end{equation}
 	for any $\mathbf{Z,W}\in \mathfrak{D}^{\nu}.$ Since $\mathbf{L}_{\perp}$ and $\mathbf{L}_{\nu}$ are totally umbilical in $\mathbf{L},$ using this fact with \eqref{d21}, \eqref{d22}, \eqref{d25} and \eqref{d26}, we conclude that $\mathbf{L}_{\perp}$ and $\mathbf{L}_{\nu}$ are totally umbilical submanifolds of $\tilde{\mathbf{L}}$, which proves statement (ii) and we are done.
 \end{proof}  
 \section*{Acknowledgments}
 The second author is thankful to CSIR for providing financial assistance in terms of JRF scholarship vide letter no.  (09/1051(0026)/2018-EMR-1).

\bigskip
.
\\
{\bfseries Statements and Declarations}
\\
The authors declare that they have no financial or non-financial interests, directly or indirectly related to the work presented for publication.
\\
\\
{\bfseries Data availability}
\\
Not applicable.
\\
\\
{\bfseries Author Contribution statement}
\\
All authors contributed equally to the paper.

\end{document}